\documentclass[]{amsart}

\usepackage{graphicx,color,hyperref}
\renewcommand{\labelenumi}{(\roman{enumi})}
\numberwithin{equation}{section}
\usepackage[initials,nobysame]{amsrefs}

\usepackage{mathtools}
\mathtoolsset{showonlyrefs=true}

\newtheorem{theorem}{Theorem}[section]
\newtheorem{proposition}[theorem]{Proposition}

\newtheorem{lemma}[theorem]{Lemma}
\theoremstyle{definition}

\theoremstyle{remark}
\newtheorem{remark}[theorem]{Remark}
\newtheorem{example}[theorem]{Example}

\newcommand{\R}{\mathbf{R}}

\begin{document}

\title{Migrating elastic flows}

\author[T.~Kemmochi]{Tomoya Kemmochi}
\address[T.~Kemmochi]{Department of Applied Physics,
Graduate School of Engineering,
Nagoya University, Furo-cho, Chikusa-ku, Nagoya, Aichi, 464-8603, Japan}
\email{kemmochi@na.nuap.nagoya-u.ac.jp}

\author[T.~Miura]{Tatsuya Miura}
\address[T.~Miura]{Department of Mathematics, Tokyo Institute of Technology, Meguro, Tokyo 152-8511, Japan}
\email{miura@math.titech.ac.jp}

\date{\today}

\begin{abstract}
  Huisken's problem asks whether there is an elastic flow of closed planar curves that is initially contained in the upper half-plane but `migrates' to the lower half-plane at a positive time.
  Here we consider variants of Huisken's problem for open curves under the natural boundary condition, and construct various migrating elastic flows both analytically and numerically.
\end{abstract}

\maketitle


\section{Introduction}

The lack of maximum principles often brings about peculiar phenomena in higher-order parabolic equations.
In the context of geometric flows for curves, the curve shortening flow is the most typical second-order flow, and it is well known that many properties are preserved along the flow, such as convexity, embeddedness, graphicality and so on.
In stark contrast, most higher-order flows possess various `positivity-losing' properties \cite{Blatt2010}.
Elastic flows are typical examples of fourth-order flows (see a survey \cite{MPP21}), which may lose positivity \cite{miura2021optimal}.
In particular, if an initial closed curve $\gamma_0$ is contained in a half-plane $H$ (or more generally in a convex set), then the curve shortening flow from $\gamma_0$ moves only within $H$, while the elastic flow is possible to protrude from $H$.
In this regard, an interesting problem is posed by G.\ Huisken (cf.\ \cite[p.118]{MPP21}) about the possibility of a `migration' phenomenon: \emph{if an immersed closed curve $\gamma_0$ is contained in the upper half-plane, then is it possible to prove that the (length-penalized) elastic flow starting from $\gamma_0$ cannot be contained in the lower half-plane at any positive time?}
Up to now this problem is still totally open.

In this paper we study some variants of Huisken's problem.
More precisely, instead of closed curves, we consider open curves under the so-called natural boundary condition, and we address both the length-penalized and length-preserving elastic flows.
In the length-preserving case, we prove that the migration phenomenon indeed occurs.
In addition, in many other cases for both length-preserving and length-penalized flows, we find out various migration phenomena through numerical computations.
To the authors' knowledge, our study would provide the first results on the migration of elastic flows.

Consider the \emph{length-preserving elastic flow}, which is defined by a smooth one-parameter family of immersed curves $\gamma:[0,1]\times[0,\infty)\to\mathbf{R}^2$ such that
\begin{align}\label{eq:flow}
  \partial_t\gamma = -2\nabla_s^2\kappa-|\kappa|^2\kappa+\lambda\kappa,
\end{align}
where $s$ denotes the arclength parameter, $\kappa=\kappa[\gamma]:=\partial_s^2\gamma$ the curvature vector ($\partial_s=|\partial_x\gamma|^{-1}\partial_x$), and $\nabla_s\psi:=\partial_s\psi-\langle \partial_s\psi,\partial_s\gamma \rangle \partial_s\gamma$ the normal derivative along $\gamma$, and $\lambda$ is the time-dependent nonlocal quantity given by
\begin{equation}
  \lambda(t) = \lambda[\gamma(\cdot,t)] := \frac{\int_{\gamma(\cdot,t)}\langle 2\nabla_s^2\kappa+\kappa^3,\kappa \rangle ds}{\int_{\gamma(\cdot,t)}|\kappa|^2ds}.
\end{equation}
The bracket $\langle \cdot,\cdot \rangle$ denotes the Euclidean inner product.
The above nonlocal equation arises as an $L^2(ds)$-gradient flow of the bending energy
\begin{equation}
  B[\gamma]:=\int_\gamma |\kappa|^2ds
\end{equation}
under the length-preserving constraint $\frac{d}{dt}L[\gamma(\cdot,t)]\equiv0$, where $L[\gamma]:=\int_{\gamma}ds$.
In addition, we impose the \emph{natural boundary condition} in which the endpoints are fixed and the curvature vanishes there: for given $p_0,p_1\in\mathbf{R}^2$,
\begin{equation}\label{eq:BC}
  \gamma(0,t)=p_0,\ \gamma(1,t)=p_1,\ \kappa(0,t)=\kappa(1,t)=0 \quad \mbox{for all $t\geq0$}.
\end{equation}
Thanks to this boundary condition the flow indeed decreases the bending energy while keeping the length fixed.
Long-time existence and sub-convergence of the flow \eqref{eq:flow} under the boundary condition \eqref{eq:BC} are proven by Dall'Acqua--Lin--Pozzi \cite{DLP2014} (see also a survey \cite{MPP21} for related results on elastic flows).

Now we state our main result.
Let $I:=(0,1)$ and $\bar{I}:=[0,1]$.
Let
$$H_\pm:=\{x\in\mathbf{R}^2 \mid \pm\langle x,e_2 \rangle \geq 0 \}$$
be the closed upper and lower half-planes, respectively, and let 
$$\Lambda_0:=\{x\in\mathbf{R}^2 \mid \langle x,e_2 \rangle =0\}$$ 
be the boundary line.
We also write the (interior) open half-planes as
$$H_{\pm}^\circ:=H_\pm\setminus\Lambda_0.$$ 
Our result asserts that if the endpoints are `pinned' on the boundary $\Lambda_0$, then an initial curve in the upper half-plane can be driven into the lower half-plane.

\begin{theorem}[Migrating elastic flow]\label{thm:main_migrating}
  There exists $c\in(0,1]$ with the following property:
  Let $L>0$ and $p_0,p_1\in \Lambda_0\subset \mathbf{R}^2$ such that $0<|p_0-p_1|< cL$.
  Then there exists a smooth solution $\gamma:\bar{I}\times[0,\infty)\to\mathbf{R}^2$ to the length-preserving elastic flow \eqref{eq:flow} of length $L$ under the natural boundary condition \eqref{eq:BC}, with the property that
  there exist $0<t_0<t_1$ such that $\gamma(I\times[0,t_0])\subset H_+^\circ$ and $\gamma(I\times[t_1,\infty))\subset H_-^\circ$.
\end{theorem}

We expect that we can take $c=1$, i.e., the smallness assumption in terms of $c$ is technical, but this is left open.
Our method strongly relies on the smallness of $c$.

Now we discuss the idea of our proof.
Our main ingenuity lies in the choice of the constraints, yielding an effective reduction of the associated variational structure; in fact, our proof is completely variational.
A recent rigorous classification of critical points under the pinned (or natural) boundary condition shows that for general $\ell:=|p_0-p_1|\in(0,L)$, the global minimizers are given by two convex arcs that lie in the half-planes $H_+$ and $H_-$, respectively, and also all the other critical points are unstable; see e.g.\ \cite{Ydcds,MYarXiv2209,MYarXiv2301}.
We will first observe that if (and in fact only if) we assume the smallness $\ell\ll L$, the critical points of second smallest energy are given by two locally-convex loops, again contained in $H_\pm$, respectively.
Then we carefully perturb the upper loop to construct an initial curve that is still contained in $H_+$, but has less energy than the loop so that the convergence limit must be one of the two arcs.
The remaining task is to prove that the flow tends to the desired lower arc in $H_-$.
To this end we show that there is an energy mountain-pass between the upper loop and the upper arc, and that this mountain-pass cannot be crossed by any small-energy elastic flow, where we again importantly use the smallness assumption $\ell\ll L$.

In the rest of this section we exhibit some open problems.
At this moment it is not clear how our problem is related to Huisken's original problem; it would be natural to expect that our boundary condition plays the role of a driving force that makes the flow stick to the boundary and hence easier to migrate.
Here, instead of Huisken's problem, we discuss our broad expectation that elastic flows are possible to migrate in more general cases under the natural boundary condition.
This expectation will be supported by our numerical computations.

The most immediate question is to ask whether we can take $c=1$ in Theorem \ref{thm:globalexistence}.
Here we already encounter the essential technical difficulty that the variational structure is so different that the second smallest critical points are laid across the two half-planes.
However, we conjecture that $c=1$ is allowable as our numerical computation suggests; even if $\frac{\ell}{L}\sim1$, we can observe a `loop-sliding' solution (as in Figure \ref{fig:constrained_asymmetric_short} below) similar to the case of $\frac{\ell}{L}\ll1$ (as in Figure \ref{fig:constrained_asymmetric_long} below).

Another focus would be on the symmetry.
The fact is that the flow constructed here starts from an asymmetric configuration, so it would also be interesting to ask whether Theorem \ref{thm:globalexistence} holds under the additional condition that the initial curve is reflectionally symmetric with respect to the vertical axis.
In fact, we could numerically find that there is a symmetric but migrating elastic flow.
The flow possesses two `loop-sliding' structures in a symmetric position, which suggests that the `loop-sliding' behavior would be one of the generic features of elastic flows.

Moreover, it is also natural to consider the length-penalized elastic flow, where $\lambda$ is just a given positive constant in \eqref{eq:flow}, as in Huisken's original problem.
Here the main analytical difficulty is that a straight segment is always a trivial global minimizer, which is a strong attractor (to which many flows converge) but approachable from both the upper and lower sides so that more delicate analysis is needed.
Notwithstanding, even for this length-penalized flow, we could numerically find various migrating examples.
On the other hand, we could also discover some examples of initial curves that fail to migrate in the length-penalized case while migrating in the length-preserved case.
Such examples are often observed in the regime that the effect of the length functional is much stronger than the bending energy, or more precisely, $\lambda\ell^2\gg1$ --- this quantity is scale invariant since the rescaling $\tilde{\gamma}(x,t):=\frac{1}{\ell}\gamma(x,\ell^4t)$ normalizes $\ell$ to be $1$ and produces the flow with $\lambda$ replaced by $\lambda\ell^2$.
In this sense our results suggest that there is a nontrivial connection between the migration behavior and the parameter $\lambda\ell^2$.
However we point out that, even from such a scaling point of view, closed curves are much harder to handle since in general the factor $\lambda$ can be normalized to $1$ just by rescaling, and hence the scaling effects essentially depend on the a priori-unknown and moving-in-time geometry of curves (unlike the given and fixed parameter $\lambda\ell^2$ in our problem).

This paper is organized as follows:
In Section \ref{sect:analytic} we give an analytic proof of Theorem \ref{thm:main_migrating}.
Section \ref{sect:numerical} exhibits various numerical computations, concerning not only the behavior corresponding to Theorem \ref{thm:main_migrating} but also other cases for which no analytical results exist.

\subsection*{Acknowledgments}

KT is supported by JSPS KAKENHI Grant Numbers 19K14590 and 21H00990.
TM is supported by JSPS KAKENHI Grant Numbers 18H03670, 20K14341, and 21H00990, and by Grant for Basic Science Research Projects from The Sumitomo Foundation.

\section{Existence of migrating elastic flows}\label{sect:analytic}

In this section we prove Theorem \ref{thm:main_migrating}, focusing on the case that $L=1$, $\ell\in(0,1)$, $p_0=(0,0)$, and $p_1=(\ell,0)$ just for notational simplicity.
This does not lose generality since our problem is invariant with respect to similarity transformations.

We first prepare some notations (including general $L>0$ for later use).
For $0\leq \ell < L$, let $A_{\ell,L}$ be the set of immersed $H^2$-Sobolev planar curves of length $L$ and fixed endpoints $p_0$ and $p_1$:
\begin{align*}
    A_{\ell,L} := \{\gamma\in H^2_\mathrm{imm}(I;\mathbf{R}^2) \mid \gamma(0)=(0,0),\ \gamma(1)=(\ell,0),\ L[\gamma]=L \},
\end{align*}
where 
$$H^2_\mathrm{imm}(I;\mathbf{R}^2) := \big\{ \gamma\in H^2(I;\mathbf{R}^2) \mid \min_{x\in\bar{I}}|\gamma'(x)|>0  \big\}.$$
Note that, by the Sobolev embedding $H^2(I)\hookrightarrow C^1(\bar{I})$, the above pointwise conditions up to first order are well defined, and also the arclength reparameterization is well defined in the sense that the resulting curve is still of class $H^2$.

In particular, for simplicity, in the case of unit length $L=1$ we write
$$A_\ell:=A_{\ell,1}.$$
In this case the arclength reparameterized curve has the same domain $I$.

\subsection{Long-time existence and convergence}

First of all we recall the long-time existence and convergence result by Dall'Acqua--Lin--Pozzi, which is arranged for our purpose.
In particular their original statement claims the last subconvergence statement in a slightly different way, but their proof immediately implies the assertion below.

\begin{theorem}[\cite{DLP2014}]\label{thm:globalexistence}
  Let $\ell\in(0,1)$.
  Let $\gamma_0:[0,1]\to\mathbf{R}^2$ be a smoothly immersed curve such that $L[\gamma_0]=1$, $\gamma_0(0)=p_0$, $\gamma_0(1)=p_1$, and $\kappa[\gamma_0](0)=\kappa[\gamma_0](1)=0$.
  Then there is a global-in-time smooth solution $\gamma:[0,1]\times[0,\infty)\to\mathbf{R}^2$ to the initial value problem,
  \begin{equation}\label{eq:elasticflow}
    \begin{cases}
      \partial_t\gamma = -2\nabla_s^2\kappa-|\kappa|^2\kappa+\lambda\kappa \quad \text{on}\ [0,1]\times[0,\infty),\\
      \gamma(x,0)=\gamma_0(x) \quad \text{for}\ x\in[0,1],\\
      \gamma(0,t)=p_0,\ \gamma(1,t)=p_1 \quad \text{for}\ t\in[0,\infty),\\
      \kappa[\gamma](0,t)=\kappa[\gamma](1,t)=0 \quad\text{for}\ t\in[0,\infty),
    \end{cases}
  \end{equation}
  where
  \begin{equation}\label{eq:lambda_nonlocal}
    \lambda(t) = \lambda[\gamma(\cdot,t)] = \frac{\int_{\gamma(\cdot,t)}\langle 2\nabla_s^2\kappa+\kappa^3,\kappa \rangle ds}{\int_{\gamma(\cdot,t)}|\kappa|^2ds}.
  \end{equation}
  Moreover, for any sequence $t_j\to\infty$ there is a subsequence $\{t_{j'}\}_{j'}$ such that, up to (arclength) reparameterization, $\gamma(\cdot,t_{j'})$ converges smoothly to a critical point $\gamma_\infty$, i.e., the curve $\gamma_\infty$ solves the following boundary value problem for some $\lambda\in\mathbf{R}$,
  \begin{equation}\label{eq:criticalpoint}
    \begin{cases}
      -2\nabla_s^2\kappa-|\kappa|^2\kappa+\lambda\kappa=0,\\
      \gamma(0)=p_0,\ \gamma(1)=p_1,\ \kappa[\gamma](0)=\kappa[\gamma](1)=0.
    \end{cases}
  \end{equation}
  The smooth convergence precisely means that $\lim_{j'\to\infty}\|\gamma(\cdot,t_{j'})-\gamma_\infty\|_{C^m([0,1];\mathbf{R}^n)}=0$ holds for any integer $m\geq0$.
\end{theorem}

\begin{remark}\label{rem:monotonicity}
  Along the flow $\frac{d}{dt}B[\gamma(\cdot,t)]\leq0$ and $L[\gamma(\cdot,t)]\equiv1$.
\end{remark}

\begin{remark}\label{rem:fullconvergence}
  Given an initial curve, if we can deduce that there is only a unique candidate of the limit critical point $\gamma_\infty$ up to reparameterization, then we automatically get the full convergence in the sense that, up to the arclength reparameterization, $\gamma(\cdot,t)\to\gamma_\infty$ smoothly as $t\to\infty$.
  In fact, if not, there are $m\in\mathbf{Z}_{\geq0}$, $\delta>0$, and $t_j\to\infty$ such that $\|\gamma(\cdot,t_j)-\gamma_\infty\|_{C^m}\geq\delta$ for all $j$.
  However the subconvergence statement implies that there is a subsequence that converges to a critical point, which is $\gamma_\infty$ by uniqueness, but this is a contradiction.
\end{remark}

\subsection{Stationary solutions}

For an immersed curve $\gamma\in H^2(I;\mathbf{R}^2)\subset C^1(\bar{I};\mathbf{R}^2)$, let $\theta=\theta[\gamma]\in H^1(I)\subset C(\bar{I})$ denote the tangential angle function defined so that
$\partial_s\gamma = (\cos\theta,\sin\theta),$
which is unique modulo $2\pi$.
Let $k=k[\gamma]\in L^2(I)$ denote the signed curvature defined by
$k=\partial_s\theta,$
and let $TC$ denote the total (signed) curvature
\begin{equation}\label{eq:TC_angle}
TC[\gamma]:=\int_\gamma kds = \theta(1)-\theta(0).    
\end{equation}

Recall that any critical point of $B$ in $A_\ell$ satisfies \eqref{eq:criticalpoint} for some $\lambda\in\mathbf{R}$, cf.\ \cite{Ydcds,Miura_LiYau} (see also \cite{MYarXiv2209}).
In addition, if $\ell=0$, then any critical point is a half-fold figure-eight elastica or its $N$ times extension \cite{Miura_LiYau}; if $\ell>0$, then any critical point is one of the upper and lower arcs $\gamma_\mathrm{arc}^{\ell,\pm,1}$, the upper and lower loops $\gamma_\mathrm{loop}^{\ell,\pm,1}$, and their suitable $N$ times extensions $\gamma_\mathrm{arc}^{\ell,\pm,N}$ and $\gamma_\mathrm{loop}^{\ell,\pm,N}$ with $N\geq2$ \cite{Ydcds}.
In the following lemma we summarize fundamental properties of the critical points which we will use later (see also Figure~\ref{fig:criticalpoints}).

\begin{figure}
    \centering
    \begin{tabular}{cc}
        \includegraphics[page=1]{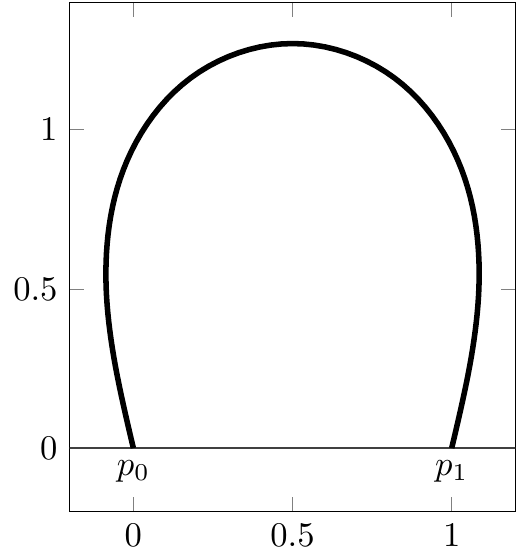} & \includegraphics[page=2]{criticalpoints.pdf} \\
        (a) upper arc $\gamma_\mathrm{arc}^{\ell,+,1}=\gamma_\mathrm{arc}^{\ell,+}$ & (b) upper loop $\gamma_\mathrm{loop}^{\ell,+,1}=\gamma_\mathrm{loop}^{\ell,+}$
    \end{tabular}
    
    \caption{Critical points; the upper arc and the upper loop.}
    \label{fig:criticalpoints}
\end{figure}

\begin{lemma}[Basic properties of critical points]\label{lem:allcriticalpoints}
  Let $\ell\in(0,1)$.
  Let $\gamma\in A_\ell$ be a solution to \eqref{eq:criticalpoint} for some $\lambda\in\mathbf{R}$.
  Then, up to reparameterization, $\gamma$ is given by one of the constant-speed smooth (analytic) curves 
  $$\gamma_\mathrm{arc}^{\ell,+,N},\gamma_\mathrm{arc}^{\ell,-,N},\gamma_\mathrm{loop}^{\ell,+,N},\gamma_\mathrm{loop}^{\ell,-,N}\in A_\ell$$
  for some integer $N\geq1$, where for simplicity we also write
  $$\text{$\gamma_\mathrm{arc}^{\ell,\pm}:=\gamma_\mathrm{arc}^{\ell,\pm,1}$ and $\gamma_\mathrm{loop}^{\ell,\pm}:=\gamma_\mathrm{loop}^{\ell,\pm,1}$},$$
  such that all the following properties hold:
  The curves $\gamma_\mathrm{arc}^{\ell,+,N}$ and $\gamma_\mathrm{loop}^{\ell,+,N}$ are the reflections of $\gamma_\mathrm{arc}^{\ell,-,N}$ and $\gamma_\mathrm{loop}^{\ell,-,N}$ through the $x$-axis, respectively, and
    \begin{align}
      & \text{$\gamma_\mathrm{arc}^{\ell,\pm}(I)\subset H_\pm^\circ$ and $\gamma_\mathrm{loop}^{\ell,\pm}(I)\subset H_\pm^\circ$,} \label{eq:cri2}\\
      & \text{$\pm\theta[\gamma_\mathrm{arc}^{\ell,\pm}](0)\in(0,\pi)$, $\pm\theta[\gamma_\mathrm{arc}^{\ell,\pm}](1)\in(-\pi,0)$, and $(\pm\theta[\gamma_\mathrm{arc}^{\ell,\pm}])'<0$ on $I$,} \label{eq:cri8}\\
      & \text{$\pm\theta[\gamma_\mathrm{loop}^{\ell,\pm}](0)\in(0,\tfrac{\pi}{2})$, $\pm\theta[\gamma_\mathrm{loop}^{\ell,\pm}](1)\in(\tfrac{3\pi}{2},2\pi)$, and $(\pm\theta[\gamma_\mathrm{loop}^{\ell,\pm}])'>0$ on $I$,} \label{eq:cri3}
    \end{align}
    and also there is $a\in(0,\tfrac{1}{2})$ such that
    \begin{align}
      & \text{$\gamma_\mathrm{loop}^{\ell,\pm}(a)=\gamma_\mathrm{loop}^{\ell,\pm}(1-a)$, $\theta[\gamma_\mathrm{loop}^{\ell,\pm}](a)\in(0,\tfrac{\pi}{2})$, and $\theta[\gamma_\mathrm{loop}^{\ell,\pm}](1-a)\in(\tfrac{3\pi}{2},2\pi)$.} \label{eq:cri7}
    \end{align}
    In addition, the total curvature and the bending energy satisfy
    \begin{align}
      & \text{$ \mp TC[\gamma_\mathrm{arc}^{\ell,\pm}]>0$ and $\pm TC[\gamma_\mathrm{loop}^{\ell,\pm}]>0$,} \label{eq:cri4}\\
      & \text{$0<B[\gamma_\mathrm{arc}^{\ell,+}]=B[\gamma_\mathrm{arc}^{\ell,-}] <  B[\gamma_\mathrm{loop}^{\ell,+}]=B[\gamma_\mathrm{loop}^{\ell,-}]$,} \label{eq:cri1}\\
      & \text{$\lim_{\ell\to0}B[\gamma_\mathrm{arc}^{\ell,\pm}] = \lim_{\ell\to0}B[\gamma_\mathrm{loop}^{\ell,\pm}]=\varpi^*:=\inf_{\gamma\in A_0}B[\gamma]$,}\label{eq:cri5}\\
      & \text{$B[\gamma_\mathrm{arc}^{\ell,\pm,N}] = N^2B[\gamma_\mathrm{arc}^{\ell,\pm}]$ and $B[\gamma_\mathrm{loop}^{\ell,\pm,N}] = N^2B[\gamma_\mathrm{loop}^{\ell,\pm}]$.} \label{eq:cri6}
    \end{align}
\end{lemma}

\begin{proof}
  All but property \eqref{eq:cri5} follow by the contents in \cite{Ydcds}.
  Property \eqref{eq:cri5} also easily follows since from the explicit formulae we deduce that as $\ell\to0$ each of $\gamma_\mathrm{arc}^{\ell,\pm}$ and $\gamma_\mathrm{loop}^{\ell,\pm}$ smoothly converges to a half-fold figure-eight elastica, which is shown to be a minimizer of $B$ in $A_0$ \cite[Proposition 2.6]{Miura_LiYau}.
\end{proof}

From the above lemma we also obtain some elementary perturbative properties, the proof of which can be safely omitted. 

\begin{lemma}\label{lem:perturb}
  Let $\ell\in(0,1)$.
  Then there exists $\varepsilon>0$ such that for any $\gamma\in A_\ell$ the following properties hold:
  \begin{itemize}
      \item[(i)] If $\|\gamma-\gamma_\mathrm{loop}^{\ell,+}\|_{C^1}\leq\varepsilon$, then $\gamma(I)\subset H_+^\circ$, $\theta[\gamma](0)\in(0,\frac{\pi}{2})$, and $\theta[\gamma](1)\in(\frac{3\pi}{2},2\pi)$.
      \item[(ii)] If $\|\gamma-\gamma_\mathrm{arc}^{\ell,-}\|_{C^1}\leq\varepsilon$, then $\gamma(I)\subset H_-^\circ$, $\theta[\gamma](0)\in(-\pi,0)$, and $\theta[\gamma](1)\in(0,\pi)$.
      \item [(iii)] If $\|\gamma-\gamma_\mathrm{loop}^{\ell,+}\|_{C^1}\leq\varepsilon$ or $\|\gamma-\gamma_\mathrm{arc}^{\ell,-}\|_{C^1}\leq\varepsilon$, then $TC[\gamma]>0$.
      \item [(iv)] If $\|\gamma-\gamma_\mathrm{loop}^{\ell,+}\|_{C^1}\leq\varepsilon$, then $\theta[\gamma](\bar{I})\subset (0,2\pi)$, and there are $a\in(0,\frac{1}{2})$ and $b\in(\frac{1}{2},1)$ such that $\gamma(a)=\gamma(b)$ and $(\theta[\gamma])'>0$ on $[a,b]$.
      In particular, $\theta[\gamma](a)\in(0,\tfrac{\pi}{2})$ and $\theta[\gamma](b)\in(\tfrac{3\pi}{2},2\pi)$.
  \end{itemize}
\end{lemma}

Notice that the last conditions on $\theta[\gamma](a)$ and $\theta[\gamma](b)$ follow since otherwise the vertical component of $\gamma$ would be strictly monotone in $[a,b]$, which contradicts $\gamma(a)=\gamma(b)$.

\subsection{Flows converging to a unique arc}

The purpose here is to give an effective sufficient condition on initial data for converging to the lower convex arc.

\begin{proposition}[Convergence to the lower arc]\label{prop:migration}
  There exists $c\in(0,1]$ with the following property:
  Let $\ell\in(0,c)$.
  If an initial curve $\gamma_0$ in Theorem \ref{thm:globalexistence} satisfies $TC[\gamma_0]>0$ and $B[\gamma_0]<B[\gamma_\mathrm{loop}^{\ell,+}]$, then the solution $\gamma(\cdot,t)$ smoothly converges, up to reparameterization, to the unique curve $\gamma_\mathrm{arc}^{\ell,-}$ as $t\to\infty$.
  In particular, there is $t_1>0$ such that $\gamma(I\times[t_1,\infty))\subset H_-^\circ$.
\end{proposition}

The proof is divided into several steps.

We first indicate that the above energy-smallness assumption already implies that the limit must be given by one of the two convex arcs.

\begin{lemma}\label{lem:uniquecriticalpoint}
  There exists $c_1\in(0,1]$ with the following property:
  Let $\ell\in(0,c_1)$.
  Let $\gamma\in A_\ell$ be a solution to \eqref{eq:criticalpoint} for some $\lambda\in\mathbf{R}$.
  If $B[\gamma]< B[\gamma_\mathrm{loop}^{\ell,\pm}]$, then up to reparameterization $\gamma$ is given by either $\gamma_\mathrm{arc}^{\ell,+}$ or $\gamma_\mathrm{arc}^{\ell,-}$.
\end{lemma}

\begin{proof}
  The arcs $\gamma_\mathrm{arc}^{\ell,\pm}$ are global minimizers among all solutions to \eqref{eq:criticalpoint}, cf.\ \eqref{eq:cri1} and \eqref{eq:cri6}.
  Hence it suffices to show that for any small $\ell\in(0,1)$ the loops $\gamma_\mathrm{loop}^{\ell,\pm}$ have the second smallest energy among all solutions to \eqref{eq:criticalpoint}.
  By \eqref{eq:cri1} and \eqref{eq:cri6} we only need to show that $4B[\gamma_\mathrm{arc}^{\ell,\pm}]\geq B[\gamma_\mathrm{loop}^{\ell,\pm}]$ for any small $\ell\in(0,1)$.
  This follows since \eqref{eq:cri5} implies that $4B[\gamma_\mathrm{arc}^{\ell,\pm}]\to4\varpi^*$ and $B[\gamma_\mathrm{loop}^{\ell,\pm}]\to\varpi^*$ as $\ell\to0$.
\end{proof}

\begin{remark}
  This result does not hold if $\ell$ is not small, since $B[\gamma_\mathrm{arc}^{\ell,\pm}]\to0$ while $B[\gamma_\mathrm{loop}^{\ell,\pm}]\to\infty$ as $\ell\to1$, and hence for any $N\geq2$ there is $\ell_N\in(0,1)$ such that if $\ell\geq\ell_N$, then $B[\gamma_\mathrm{arc}^{\ell,\pm,N}]=N^2B[\gamma_\mathrm{arc}^{\ell,\pm}]< B[\gamma_\mathrm{loop}^{\ell,\pm}]$.
\end{remark}

Now we turn to the main argument for detecting the lower arc.
The first ingredient is very simple but provides a universal energy control along the flow below the energy level $B[\gamma_\mathrm{loop}^{\ell,\pm}]$.

\begin{lemma}\label{lem:upperbound}
  There exists $c_2\in(0,1]$ such that if $\ell\in(0,c_2)$, then $B[\gamma_\mathrm{loop}^{\ell,\pm}]\leq 2\varpi^*$.
\end{lemma}

\begin{proof}
  This follows since $B[\gamma_\mathrm{loop}^{\ell,\pm}]\to\varpi^*$ as $\ell\to0$, cf.\ \eqref{eq:cri5}.
\end{proof}

On the other hand, the next lemma shows that there is a certain energy `mountain-pass', which is fortunately larger than the above upper bound $2\varpi^*$.

\begin{lemma}\label{lem:lowerbound}
  Let $A'_\ell:=\{\gamma\in A_\ell \mid TC[\gamma]=0 \}$.
  Then
  \begin{equation}\label{eq:04}
    \liminf_{\ell\to0}\Big(\inf_{\gamma\in A'_\ell}B[\gamma]\Big) \geq 4\varpi^*.
  \end{equation}
\end{lemma}

\begin{proof}
  Notice first that by an easy construction of test curves, e.g.\ suitable odd extensions of circular arcs, we have $\sup_{\ell\in(0,1)}\inf_{\gamma\in A'_\ell}B[\gamma] =:M <\infty$.

  We first prove that
  \begin{equation}\label{eq:06}
    \liminf_{\ell\to0}\Big(\inf_{\gamma\in A'_\ell}B[\gamma]\Big) \geq \inf_{\gamma\in A'_0}B[\gamma],
  \end{equation}
  Let $\ell_j\to0$ be any sequence of positive numbers.
  For each $j$ we may take an arclength parameterized curve $\gamma_{\ell_j}\in A_{\ell_j}$ such that
  $$B[\gamma_{\ell_j}] \leq \inf_{\gamma\in A'_{\ell_j}}B[\gamma] +\frac{1}{j} \leq M+1.$$
  Since $|\gamma_{\ell_j}'|\equiv1$, we have
  $$\max_{s\in\bar{I}}|\gamma_{\ell_j}(s)|=\max_{s\in\bar{I}} \left| \gamma_{\ell_j}(0)+\int_0^s\gamma_{\ell_j}' \right| \leq |\gamma_{\ell_j}(0)|+\int_0^1|\gamma_{\ell_j}'| =1,$$
  and also
  $\|\gamma_{\ell_j}''\|_2^2 = B[\gamma_{\ell_j}] \leq M+1.$
  Therefore the sequence $\{\gamma_{\ell_j}\}_j$ is $H^2$-bounded, and hence admits a subsequence that converges $H^2$-weakly and $C^1$-strongly to some $\gamma_0\in H^2(I;\mathbf{R}^2)$.
  By the $C^1$-convergence we have $|\gamma_0'|\equiv1$ and $\gamma_0\in A'_0$.
  By the $H^2$-weak convergence we have the lower semicontinuity
  $$\liminf_{j\to\infty}B[\gamma_{\ell_j}] = \liminf_{j\to\infty} \|\gamma_{\ell_j}''\|_2^2 \geq \|\gamma_0''\|_2^2= B[\gamma_0],$$
  which implies that
  $$\liminf_{\ell_j\to0}\Big(\inf_{\gamma\in A'_{\ell_j}}B[\gamma]\Big)\geq \liminf_{j\to\infty}\Big( B[\gamma_{\ell_j}] -\frac{1}{j}\Big) \geq B[\gamma_0] \geq \inf_{\gamma\in A'_0}B[\gamma].$$
  Since $\ell_j\to0$ is arbitrary, we obtain \eqref{eq:06}.

  Now we prove that
  \begin{equation}\label{eq:07}
    \inf_{\gamma\in A_0'}B[\gamma] \geq 4\varpi^*.
  \end{equation}
  Let $\gamma\in A_0'$.
  Up to reparameterization we may assume that $|\gamma'|\equiv1$.
  Since $\gamma(0)=\gamma(1)$, and since $\gamma'(0)=\gamma'(1)$ by $TC[\gamma]=0$ (and $|\gamma'|\equiv1$), we can regard $\gamma$ as a closed $H^2$-curve of rotation number zero, which has at least one self-intersection by Hopf's Umlaufsatz.
  Therefore we have $B[\gamma]\geq 4\varpi^*$ by \cite[Theorem 1.2]{MRacv} or \cite[Theorem 1.1]{Miura_LiYau}, and hence obtain \eqref{eq:07}.
  Combining \eqref{eq:06} and \eqref{eq:07}, we complete the proof.
\end{proof}

\begin{remark}
  In fact some stronger properties hold in the proof.
  For example a standard direct method implies that the infimum $\inf_{\gamma\in A'_\ell}B[\gamma]$ is always attained.
  In addition, $\inf_{\gamma\in A_0'}B[\gamma]$ is not only bounded below but equal to $4\varpi^*$, since the one-fold figure-eight elastica can be regarded as an element of $A_0'$ by opening the curve at the cross point.
  However we do not need these facts in this paper.
\end{remark}

This implies that a small-energy flow needs to preserve the sign of $TC$.

\begin{lemma}\label{lem:signpreservation}
  There exists $c_3\in(0,1]$ with the following property:
  Let $\ell\in(0,c_3)$.
  If an initial curve $\gamma_0$ in Theorem \ref{thm:globalexistence} satisfies $TC[\gamma_0]>0$ and $B[\gamma_0]<3\varpi^*$, then any limit critical point $\gamma_\infty$ in Theorem \ref{thm:globalexistence} satisfies $TC[\gamma_\infty]\geq0$.
\end{lemma}

\begin{proof}
  By Lemma \ref{lem:lowerbound} we can choose $c_3\in(0,1]$ such that if $\ell\in(0,c_3)$, then
  \begin{equation}\label{eq:05}
    \inf_{\gamma\in A'_\ell}B[\gamma] \geq 3\varpi^*.
  \end{equation}
  Let $\ell\in(0,c_3)$.
  We prove the contrapositive.
  Suppose that there is a sequence $t_j\to\infty$ such that, up to reparameterization, the flow $\gamma(\cdot,t_j)$ converges to some critical point $\gamma_\infty$ such that $TC[\gamma_\infty]<0$.
  In particular, for some $t_{j_0}>0$ we have $TC[\gamma(\cdot,t_{j_0})]<0$.
  Since $TC[\gamma_0]>0$ and the quantity $TC[\gamma(\cdot,t)]$ continuously depends on $t$, there exists $t_*\in(0,t_{j_0})$ such that $TC[\gamma(\cdot,t_*)]=0$, and hence $\gamma(\cdot,t_*)\in A_\ell'$.
  By \eqref{eq:05} we have $B[\gamma(\cdot,t_*)]\geq 3\varpi^*$, and hence $B[\gamma_0]\geq 3\varpi^*$ by monotonicity, cf.\ Remark \ref{rem:monotonicity}.
\end{proof}

We are now in a position to prove Proposition \ref{prop:migration}.

\begin{proof}[Proof of Proposition \ref{prop:migration}]
  Let $c:=\min\{c_1,c_2,c_3\}$, cf.\ Lemma \ref{lem:uniquecriticalpoint}, Lemma \ref{lem:upperbound}, and Lemma \ref{lem:signpreservation}.
  By monotonicity in Remark \ref{rem:monotonicity} any limit critical point $\gamma_\infty$ satisfies $B[\gamma_\infty]\leq B[\gamma_0]<B[\gamma_\mathrm{loop}^{\ell,+}]$,
  and hence by Lemma \ref{lem:uniquecriticalpoint}, up to reparameterization, we have
  \begin{equation}\label{eq:1222-1}
      \gamma_\infty\in\{\gamma_\mathrm{arc}^{\ell,+},\gamma_\mathrm{arc}^{\ell,-}\}.
  \end{equation}
  In addition, by Lemma \ref{lem:upperbound} the initial curve satisfies $B[\gamma_0]< B[\gamma_\mathrm{loop}^{\ell,+}] \leq 2\varpi^*$, and hence now Lemma \ref{lem:signpreservation} is applicable to deduce that any limit critical point $\gamma_\infty$ needs to satisfy
  \begin{equation}\label{eq:1222-2}
      TC[\gamma_\infty]\geq0.
  \end{equation}
  By \eqref{eq:cri4}, \eqref{eq:1222-1}, and \eqref{eq:1222-2}, we deduce that the limit curve is uniquely determined up to reparamaterization by $\gamma_\infty=\gamma_\mathrm{arc}^{\ell,-}$.
  Thus the desired full convergence follows, cf.~Remark \ref{rem:fullconvergence}.
  The last lower half-plane property $\gamma(I\times[t_1,\infty))\subset H_-^\circ$ also follows by the obtained convergence and Lemma \ref{lem:perturb} (ii).
\end{proof}

\subsection{Choice of well-prepared initial data}

The remaining task is to prove that we can choose an initial curve that satisfies the assumptions in Proposition \ref{prop:migration} and also is contained in the upper half-plane.

\begin{lemma}[Well-prepared initial data]\label{lem:initialcurve}
  Let $\ell\in(0,1)$.
  There exists
  \begin{align}
      & \text{$\gamma\in A_\ell \cap C^\infty(\bar{I};\mathbf{R}^2)$}, \label{eq:ini_assertion_0}
  \end{align}
  such that the following properties hold:
  \begin{align}
      & \text{$k[\gamma](0)=k[\gamma](1)=0$,} \label{eq:ini_assertion_1}\\
      & \text{$\gamma(I)\subset H_+^\circ$, $\theta[\gamma](0)\in(0,\tfrac{\pi}{2})$, $\theta[\gamma](1)\in(\tfrac{3\pi}{2},2\pi)$,} \label{eq:ini_assertion_2}\\
      & \text{$B[\gamma]<B[\gamma_\mathrm{loop}^{\ell,+}]$,} \label{eq:ini_assertion_3}\\
      & \text{$TC[\gamma]>0$.} \label{eq:ini_assertion_4}
  \end{align}
\end{lemma}

\begin{proof}
  Write $\gamma_*:=\gamma_\mathrm{loop}^{\ell,+}$ for simplicity.
  Let us first recall a known perturbation of $\gamma_*$ within $A_\ell$ that decreases both energy $B$ and length $L$, cf.\ \cite{MYarXiv2301}.
  Since the curvature is of the form $k[\gamma_*](x)=\Phi(x)$ for $x\in\bar{I}$ with some odd, antiperiodic, analytic function $\Phi:\R\to\R$ such that $\Phi(0)=0$, $\Phi(x+1)=-\Phi(x)$, and $\Phi(x)=\Phi(1-x)$, cf.\ \cite{Ydcds}, the odd extension of $\gamma_*$ defined by $\gamma_*(x)=-\gamma_*(-x)$ for $x\in[-1,0]$ is again analytic on $[-1,1]$.
  Since in addition $\theta[\gamma_*](0)\in(0,\tfrac{\pi}{2})$ and $\theta[\gamma_*](1)\in(\tfrac{3}{2}\pi,2\pi)$, for any large integer $j\geq1$ we have $|\gamma_*(-\tfrac{1}{j})-\gamma_*(1-\tfrac{1}{j})|>\ell$ and $|\gamma_*(0)-\gamma_*(1-\tfrac{1}{j})|<\ell$ and hence there is $c_j\in(0,\frac{1}{j})$ such that 
  $|\gamma_*(-c_j)-\gamma_*(1-\tfrac{1}{j})|=\ell$ (see also the proof of \cite[Theorem 2.8]{MYarXiv2301}).
  Therefore, there is a suitable orientation-preserving isometric transformation $Q:\R^2\to\R^2$ such that 
  $$Q\gamma_*|_{[-c_j,1-\frac{1}{j}]}\in A_{\ell,1-\frac{1}{j}+c_j}\cap C^\infty(\bar{I};\mathbf{R}^2),$$ 
  $L\big[Q\gamma_*|_{[-c_j,1-\frac{1}{j}]}\big]=1-\frac{1}{j}+c_j<1=L[\gamma_*]$, and $B\big[Q\gamma_*|_{[-c_j,1-\frac{1}{j}]}\big]<B[\gamma_*]$ by our construction and the above properties of $\Phi$.
  
  In summary, we can construct a constant-speed smooth curve
  \begin{align}
      & \text{$\gamma\in A_{\ell,L[\gamma]}\cap C^\infty(\bar{I};\mathbf{R}^2)$ with $|\gamma'|\equiv L[\gamma]$}, \label{eq:ini0}
  \end{align}
  such that
  \begin{align}
      & \text{$\|\gamma-\gamma_*\|_{C^1}< \varepsilon$ with $\varepsilon>0$ as in Lemma \ref{lem:perturb},} \label{eq:ini1}\\
      & \text{$L[\gamma]<1$ and $B[\gamma]<B[\gamma_*]$.} \label{eq:ini2}
  \end{align}
  In particular, Lemma \ref{lem:perturb} (i) implies that
  \begin{align}
     & \text{$\gamma(I)\subset H_+^\circ$, $\theta[\gamma](0)\in(0,\tfrac{\pi}{2})$, $\theta[\gamma](1)\in(\tfrac{3\pi}{2},2\pi)$,} \label{eq:ini31}
  \end{align}
  Lemma \ref{lem:perturb} (iii) implies that
  \begin{align}\label{eq:ini32}
      & \text{$TC[\gamma]>0$,}
  \end{align}
  and Lemma \ref{lem:perturb} (iv) implies that there are $a,b$ with $0<a<b<1$ such that
  \begin{align}
      &\text{$\gamma(a)=\gamma(b)$, $\gamma((a,b))-\gamma(a)\subset H_+^\circ$, $\theta[\gamma](a)\in(0,\tfrac{\pi}{2})$, $\theta[\gamma](b)\in(\tfrac{3\pi}{2},2\pi)$.} \label{eq:ini3}
  \end{align}
  
  Moreover, after a cut-off modification around the endpoints and also the self-intersection points, and constant-speed reparameterization, we may additionally assume that there is a small $\delta\in(0,\frac{1}{2}\min\{a,b-a,1-b\})$ such that
  \begin{align}
      & \text{$\gamma''\equiv0$ on $[0,\delta]\cup[1-\delta,1]$} \label{eq:ini4}\\
      & \text{$\gamma''\equiv0$ on $[a-\delta,a+\delta]\cup[b-\delta,b+\delta]$.} \label{eq:ini4'}
  \end{align}
  Indeed, let $\gamma_\delta:=(1-\eta_\delta)\gamma+\eta_\delta\xi$, where $x_0\in[0,1]$ and $\xi(x):=\gamma(x_0)+(x-x_0)\gamma'(x_0)$ and $\eta_\delta(x):=\eta(x/\delta)$ with any fixed nonincreasing cut-off function $\eta\in C^\infty(\mathbf{R})$ such that $\eta\equiv1$ on $[-1,1]$ and $\eta\equiv0$ on $\R\setminus(-2,2)$.
  Then $\gamma_\delta$ satisfies \eqref{eq:ini0} for any $\delta>0$.
  In addition, since $\gamma_\delta\to\gamma$ in $H^2(I)\cap C^1(\bar{I})$, all the properties in \eqref{eq:ini0}--\eqref{eq:ini3} are remained true even if $\gamma$ is replaced by the constant-speed reparameterization of $\gamma_\delta$ for any small $\delta>0$ (up to suitably redefining $a,b$).
  Thus we may assume \eqref{eq:ini4} and \eqref{eq:ini4'} by performing the above modification at the four points $x_0=0,a,b,1$ with sufficiently small $\delta$.

  Now we rescale the `loop-part' $\gamma|_{[a,b]}$ of the above $\gamma$ with a scaling factor $\sigma\geq1$, i.e., replace the part $\gamma|_{[a,b]}$ with $\sigma(\gamma|_{[a,b]}-\gamma(a))+\gamma(a)$ and reparameterize the entire curve to be of constant speed, to produce a family of constant-speed smooth curves
  \begin{align}
      & \text{$\gamma_\sigma \in A_{\ell,L[\gamma_\sigma]}\cap C^\infty(\bar{I};\mathbf{R}^2)$ with $|\gamma_\sigma'|\equiv L[\gamma_\sigma]$,} \label{eq:ini5}
  \end{align}
  where in particular the smoothness is ensured by the fact that $\gamma$ is completely flat near the self-intersection points $a$ and $b$, cf.\ \eqref{eq:ini4'},
  such that the family $\{\gamma_\sigma\}_{\sigma\geq1}$ has the properties that
  \begin{align}
      & \text{each $\gamma_\sigma$ satisfies all
      \eqref{eq:ini31}--\eqref{eq:ini4'} with $a,b,\delta$ dependent on $\sigma\geq1$,} \label{eq:ini10}\\
      & \text{$\sigma\mapsto L[\gamma_\sigma]$ is continuous, $L[\gamma_1]<1$, and $\lim_{\sigma\to\infty}L[\gamma_\sigma]=\infty$,} \label{eq:ini8}\\
      & \text{$B[\gamma_\sigma]< B[\gamma_*]$ for all $\sigma\geq1$.} \label{eq:ini7}
  \end{align}
  Note that by \eqref{eq:ini3} the loop-part is upward and hence the upper half-plane property in \eqref{eq:ini31} holds even for large $\sigma\geq1$.
  Property \eqref{eq:ini8} follows by the scaling property $L[\gamma_\sigma]-L[\gamma_1]=(\sigma-1) L[\gamma|_{[a,b]}]$ and by $L[\gamma_1]=L[\gamma]<1$, cf.\ \eqref{eq:ini2}.
  Property \eqref{eq:ini7} similarly follows by $B[\gamma_\sigma]-B[\gamma_1]=(\sigma^{-1}-1)B[\gamma|_{[a,b]}]$ and by $B[\gamma_1]=B[\gamma]<B[\gamma_*]$, cf. \eqref{eq:ini2}.
  
  Now, by \eqref{eq:ini8} we can pick a suitable $\sigma_0>1$ such that
  \begin{align}\label{eq:ini9}
      & L[\gamma_{\sigma_0}] =1,
  \end{align}
  and the curve $\gamma_{\sigma_0}$ satisfies all the desired properties.
  Indeed, \eqref{eq:ini_assertion_0} follows by \eqref{eq:ini5} with \eqref{eq:ini9}.
  All the other properties can be obtained through \eqref{eq:ini10}; more precisely, \eqref{eq:ini_assertion_1} follows by\eqref{eq:ini4'}, \eqref{eq:ini_assertion_2} follows by \eqref{eq:ini31}, \eqref{eq:ini_assertion_3} follows by \eqref{eq:ini7}, and \eqref{eq:ini_assertion_4} follows by \eqref{eq:ini32}.
\end{proof}

Now we complete the proof of our main theorem.

\begin{proof}[Proof of Theorem \ref{thm:main_migrating}]
  Choose an initial curve $\gamma_0$ as in Lemma \ref{lem:initialcurve} to the flow \eqref{eq:flow} under the boundary condition in \eqref{eq:BC}.
  Note that by \eqref{eq:ini_assertion_0} and \eqref{eq:ini_assertion_1} this curve is indeed compatible to \eqref{eq:BC}.
  Then by Theorem \ref{thm:globalexistence} the flow has a unique smooth solution.
  By \eqref{eq:ini_assertion_2} and smoothness of $\gamma$, there exists a small $t_0>0$ such that $\gamma(I\times[0,t_0])\subset H_+^\circ$.
  On the other hand, thanks to \eqref{eq:ini_assertion_3} and \eqref{eq:ini_assertion_4}, Proposition \ref{prop:migration} implies that there exists a large $t_1>0$ such that $\gamma(I\times[t_1,\infty))\subset H_-^\circ$.
\end{proof}

\section{Numerical examples of migrating elastic flows}\label{sect:numerical}

In this section, we numerically investigate both the length-preserving and the length-penalized elastic flows.
Throughout this section, we fix $\ell=1$.

The following examples are computed by almost the same scheme proposed in \cite{KMSarXiv2208}, which ensures that numerical solutions satisfy $\frac{dB}{dt} \le 0$.
The constraint $\frac{dL}{dt} = 0$ is also satisfied for the length-preserving case.
Although the method of \cite{KMSarXiv2208} is developed for the gradient flows of planar closed curves, it is not difficult to apply the machinery to the case of our boundary condition \eqref{eq:BC}.

\subsection{Length-preserving case}

We first observe the dynamics of the elastic flow \eqref{eq:elasticflow} with $\lambda$ given by \eqref{eq:lambda_nonlocal} so that the length is preserved.
In the first two examples (Examples~\ref{ex:constrained_asymmetric_long}, \ref{ex:constrained_asymmetric_short}), we consider the flow with asymmetric initial curves, and a symmetric flow is addressed in the last example (Example~\ref{ex:constrained_symmetric}).

\begin{example}[Asymmetric flow with long length]
\label{ex:constrained_asymmetric_long}
    We gave the initial curve by
    \begin{equation}
        \gamma_0(x) =
        \begin{pmatrix}
            a \sin(\pi x) + b \sin(2\pi x) + c\sin(3\pi x) + x \\
            d \sin(\pi x)
        \end{pmatrix},
        \qquad x \in [0,1]
        \label{eq:initial-asymmetric}
    \end{equation}
    with $b=1/\pi, c=0.4/\pi, d=1$, and  $a = 2b - 3c + 1/\pi$ so that the boundary condition is satisfied.
    The length is $L[\gamma_0] \approx 2.729$ and the bending energy is $B[\gamma_0] \approx 37.72$.

    The numerical result is illustrated in Figure~\ref{fig:constrained_asymmetric_long}.
    One observes that the curve starts migration after $t=0.1$ with loop-sliding behavior.
    Since in the proof of Theorem \ref{thm:main_migrating} we have also constructed an initial curve of an asymmetric loop, we expect that our theoretical solution also behaves like Figure~\ref{fig:constrained_asymmetric_long}.
    
    \begin{figure}
        \centering
        \includegraphics[page=1,width=\linewidth]{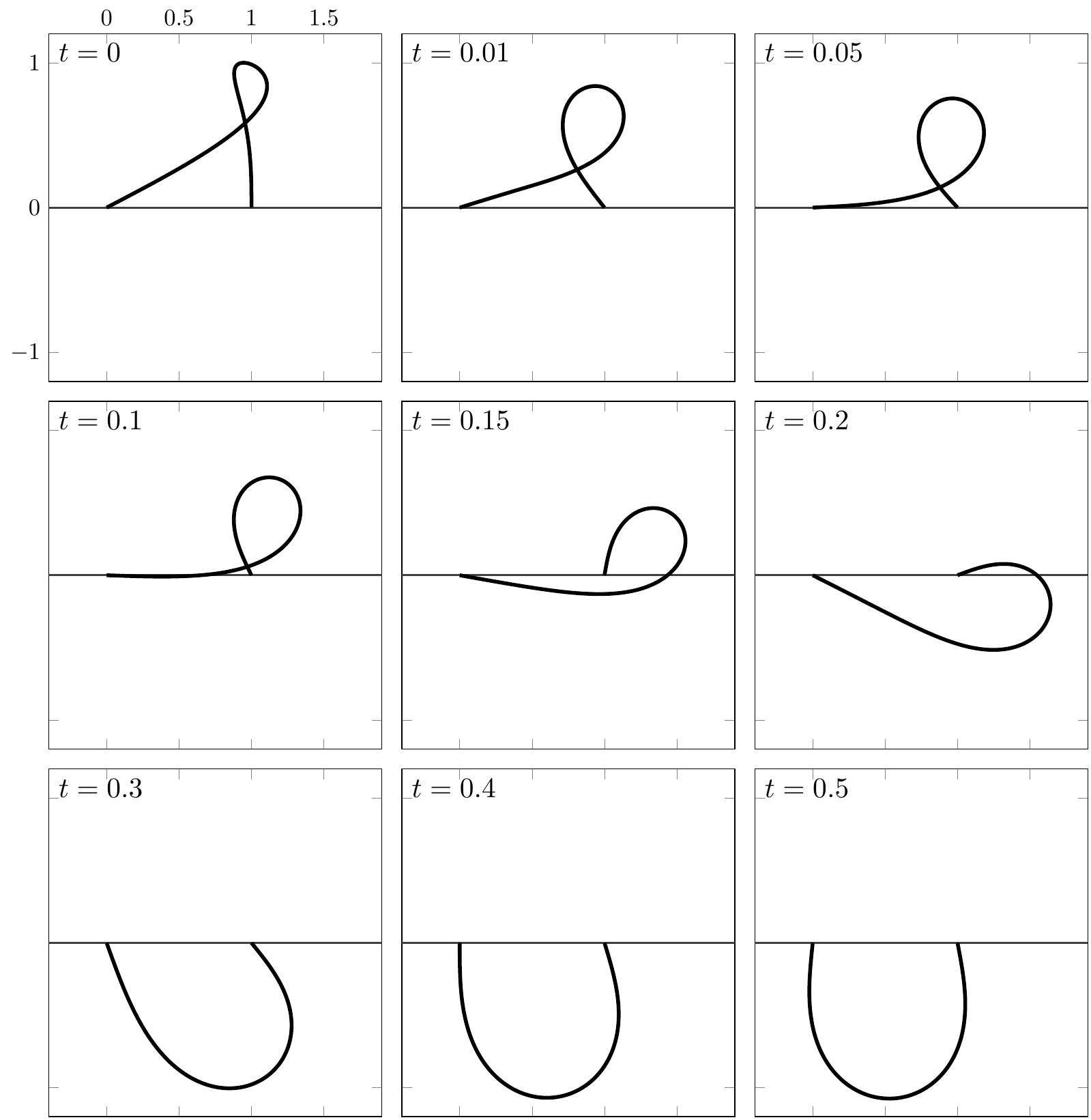}
        \caption{Numerical result of Example~\ref{ex:constrained_asymmetric_long}}
        \label{fig:constrained_asymmetric_long}
    \end{figure}
\end{example}

\begin{example}[Asymmetric flow with short length]
\label{ex:constrained_asymmetric_short}
    We gave the initial curve by \eqref{eq:initial-asymmetric} with $b=0.6/\pi, c=0.2/\pi, d=0.1$, and  $a = 2b - 3c + 1/\pi$ so that the boundary condition is satisfied.
    The length is $L[\gamma_0] \approx 1.117$, which is close to $\ell=1$, and the bending energy is $B[\gamma_0] \approx 1503$.

    The numerical result is illustrated in Figure~\ref{fig:constrained_asymmetric_short}.
    Even though this case may not be covered by Theorem~\ref{thm:main_migrating}, the migrating phenomenon is observed.
    Indeed, the loop first slides to the right and then comes untied by turning around the right endpoint after $t \approx 9.0 \cdot 10^{-5}$.
    The right part of the curve thus migrates to $H_-$, and simultaneously the left tail is waved in reaction to the untying procedure and once protrudes from $H_-$, but eventually the whole curve migrates to $H_-$.
    This result suggests that Theorem~\ref{thm:main_migrating} itself would hold even for $c \approx 1$, but the dynamical mechanism would be much more involved.
    
    \begin{figure}
        \centering
        \includegraphics[page=2,width=\linewidth]{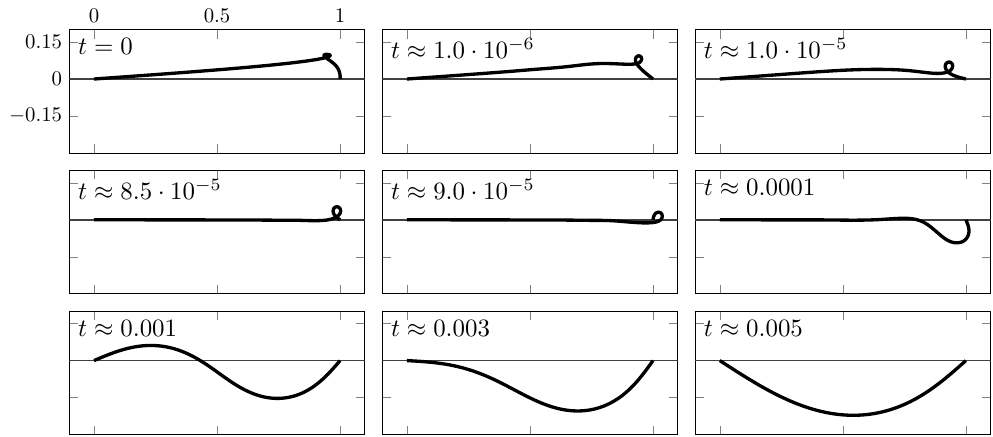}
        \caption{Numerical result of Example~\ref{ex:constrained_asymmetric_short}}
        \label{fig:constrained_asymmetric_short}
    \end{figure}
\end{example}

\begin{example}[Symmetric flow]
\label{ex:constrained_symmetric}

    We gave the initial curve by 
    \begin{equation}
        \gamma_0(x) =
        \begin{pmatrix}
            x - \frac{\sin(2\pi x)}{2\pi} + a( -2\sin(2\pi x) + \sin(4\pi x) ) \\
            b \sin(\pi x) + c \sin(3\pi x) + d \sin(5\pi x)
        \end{pmatrix},
        \qquad x \in [0,1]
        \label{eq:initial-symmetric}
    \end{equation}
    with $a = 0.15, b = 0.3, c = 0.2, d = 0.05$.
    This curve is symmetric and satisfies the boundary condition \eqref{eq:BC}.
    The length is $L[\gamma_0] \approx 2.725$ and the bending energy is $B[\gamma_0] \approx 95.58$.
    
    The numerical result is illustrated in Figure~\ref{fig:constrained_symmetric}.
    The `horizontal' part goes down to $H_-$ in the beginning.
    After $t = 0.004$, the loops come untied and the curve migrates to $H_-$.
    Although our proof of Theorem~\ref{thm:main_migrating} relies on constructing an asymmetric initial curve, the numerical result here suggests that the same assertion would hold true even under the additional assumption that the initial curve is reflectionally symmetric. 
    It would also be interesting to consider how this kind of loop-sliding behavior can be used for constructing other types of nontrivial motions.

    \begin{figure}
        \centering
        \includegraphics[page=3,width=\linewidth]{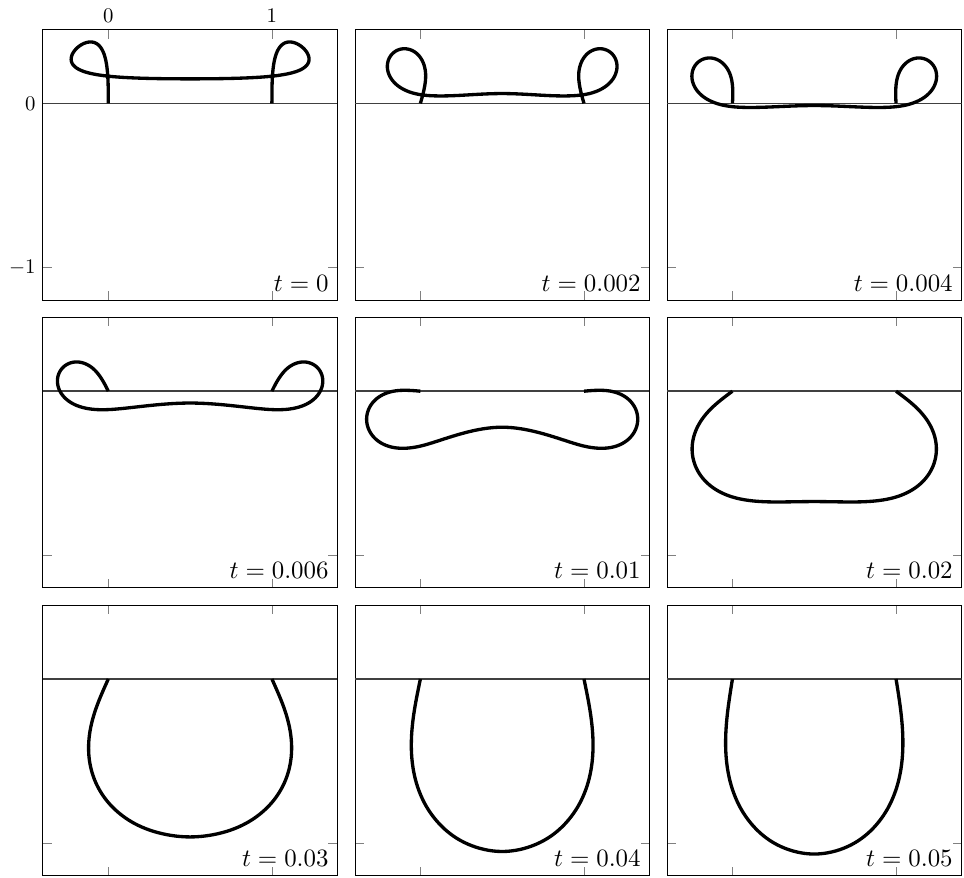}
        \caption{Numerical result of Example~\ref{ex:constrained_symmetric}}
        \label{fig:constrained_symmetric}
    \end{figure}

\end{example}

\subsection{Length-penalized case}

We next observe the elastic flow \eqref{eq:elasticflow} with a \emph{constant} $\lambda$ with the same initial curves as in the previous examples.
In the first two examples (Examples~\ref{ex:unconstrained_asymmetric_long}, \ref{ex:unconstrained_asymmetric_short}), we consider the flow with Asymmetric initial curves and observe whether there are any differences from the length-preserving case.
In the last example (Example~\ref{ex:unconstrained_symmetric}), we observe that, even for the same initial curve, the dynamics is strongly dominated by the penalization parameter $\lambda$.

\begin{example}[Asymmetric flow with small $\lambda$]
\label{ex:unconstrained_asymmetric_long}

    We gave the same initial curve as in Example~\ref{ex:constrained_asymmetric_long} and set $\lambda = 9$.
    The numerical result is illustrated in Figure~\ref{fig:unconstrained_asymmetric_long} and the migrating behavior is observed.
    The dynamics is similar to that of Example~\ref{ex:constrained_asymmetric_long} before migrating, and after $t \approx 0.24$, the curve shrinks rapidly to the line segment, which is the trivial critical point of $B + \lambda L$.
    This example suggests that the migrating phenomenon occurs even for the length-penalized flow.

    \begin{figure}
        \centering
        \includegraphics[page=4,width=\linewidth]{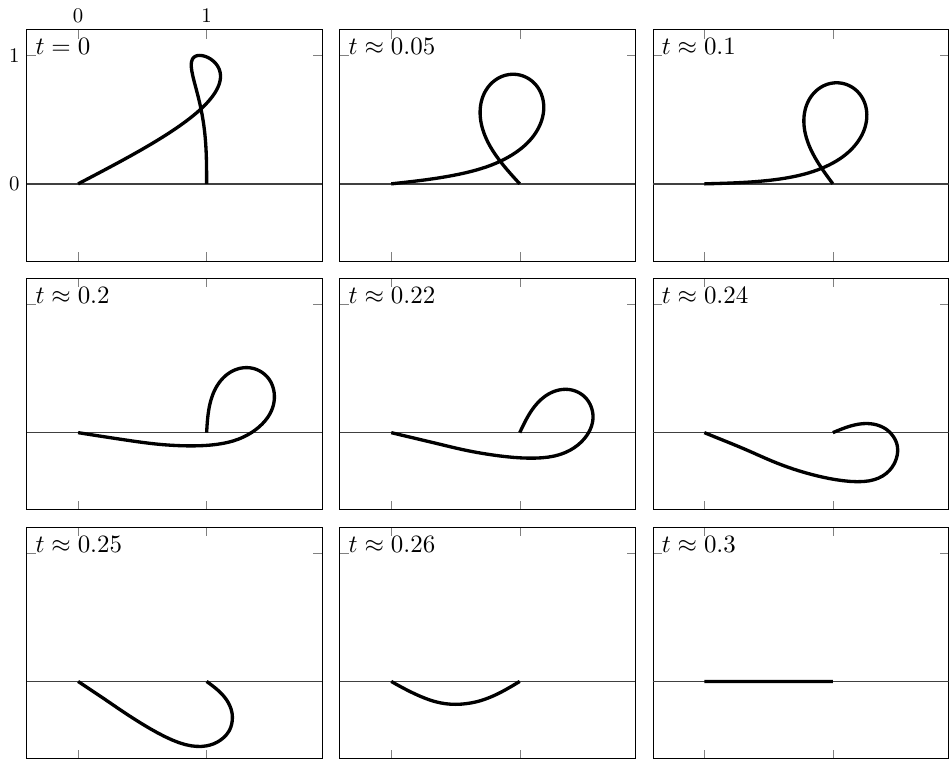}
        \caption{Numerical result of Example~\ref{ex:unconstrained_asymmetric_long}}
        \label{fig:unconstrained_asymmetric_long}
    \end{figure}
    
\end{example}

\begin{example}[Asymmetric flow with large $\lambda$]
\label{ex:unconstrained_asymmetric_short}

    We gave the same initial curve as in Example~\ref{ex:constrained_asymmetric_short} and set $\lambda = 900$.
    In this case, the loop is smaller than the previous case as $\lambda$ is large, and it slides to the right as in Example~\ref{ex:constrained_asymmetric_short}.
    After the loop-sliding behavior, one would expect that the curve migrates to $H_-$ as in Example~\ref{ex:constrained_asymmetric_short}.
    
    However, somewhat interestingly, this is not the case in our numerical computation.
    Indeed,     Figure~\ref{fig:unconstrained_asymmetric_short-rescaled} shows the vertically rescaled snapshot of the solution, suggesting that the curve converges while intersecting with the horizontal axis.
    Of course this might be just an effect of numerical errors, but this result suggests at least that it is substantially delicate whether the migration occurs in the length-penalized case compared to the length-preserved case. 
    
    \begin{figure}
        \centering
        \includegraphics[page=5,width=\linewidth]{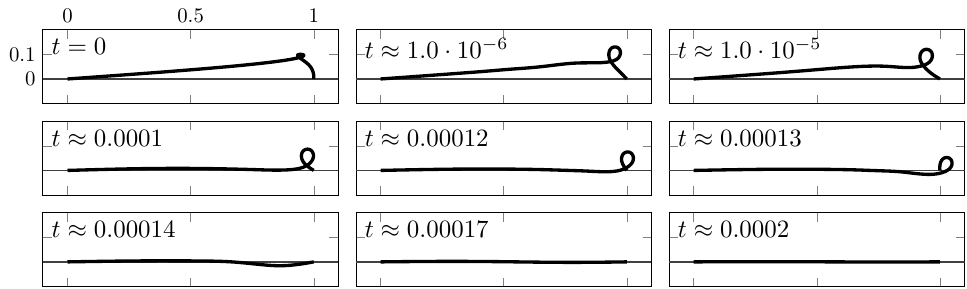}
        \caption{Numerical result of Example~\ref{ex:unconstrained_asymmetric_short}}
        \label{fig:unconstrained_asymmetric_short}
    \end{figure}

    \begin{figure}
        \centering
        \includegraphics[page=6,width=\linewidth]{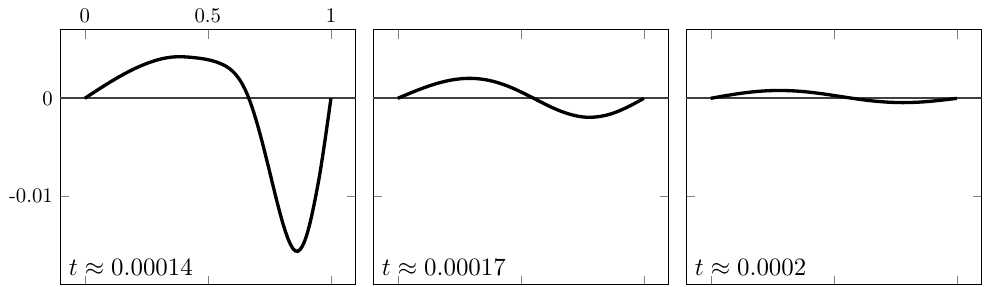}
        \caption{Example~\ref{ex:unconstrained_asymmetric_short} (vertically rescaled)}
        \label{fig:unconstrained_asymmetric_short-rescaled}
    \end{figure}
\end{example}

\begin{example}[Symmetric flow with different $\lambda$]
\label{ex:unconstrained_symmetric}

    In the last example, we gave the initial curve as in Example \ref{ex:constrained_symmetric}.
    We computed numerical solutions for $\lambda = 16$ and $\lambda = 900$ and the results are illustrated in Figures~\ref{fig:unconstrained_symmetric-1} and \ref{fig:unconstrained_symmetric-2}, respectively.
    
    One observes very different behaviors for the different parameters.
    For the smaller $\lambda$, the flow behaves like the length-preserved case until the loops come untied, and also eventually migrates to $H_-$, but the curve converges not to an arc but to a segment while shrinking the length.
    On the other hand, for the bigger $\lambda$, the shrinking behavior is more dominant and this prevents the solution from migrating.
    Indeed, the solution does not migrate to $H_-$ as indicated in Figure~\ref{fig:unconstrained_symmetric-2-rescaled}, which illustrates the vertically rescaled snapshot. 
    The behavior observed here would be quite compatible with Example \ref{ex:unconstrained_asymmetric_short} (up to a symmetric extension).

    \begin{figure}
        \centering
        \includegraphics[page=7,width=\linewidth]{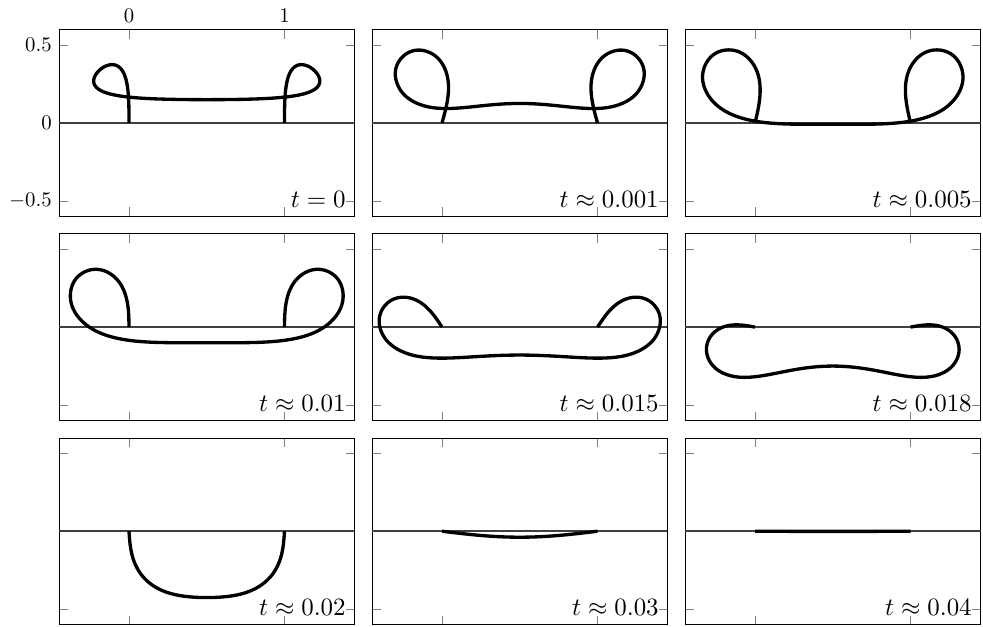}
        \caption{Numerical result of Example~\ref{ex:unconstrained_symmetric} ($\lambda = 16$)}
        \label{fig:unconstrained_symmetric-1}
    \end{figure}

    \begin{figure}
        \centering
        \includegraphics[page=8,width=\linewidth]{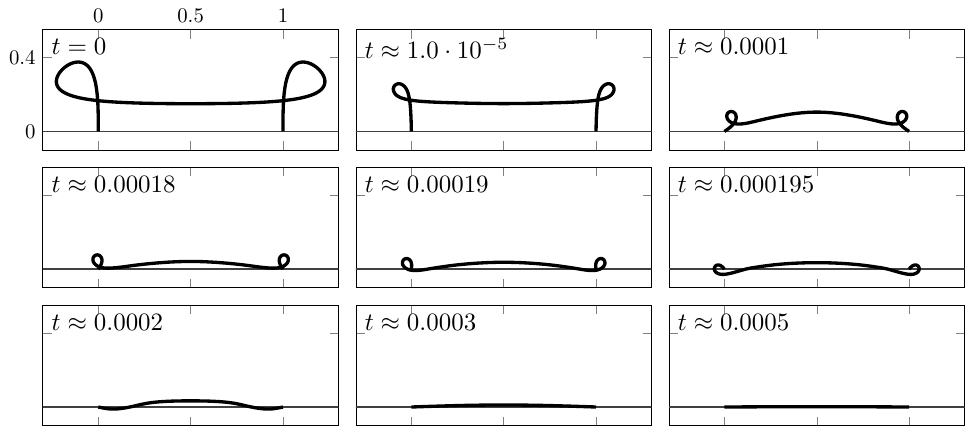}
        \caption{Numerical result of Example~\ref{ex:unconstrained_symmetric} ($\lambda = 900$)}
        \label{fig:unconstrained_symmetric-2}
    \end{figure}
    
    \begin{figure}
        \centering
        \includegraphics[page=9,width=\linewidth]{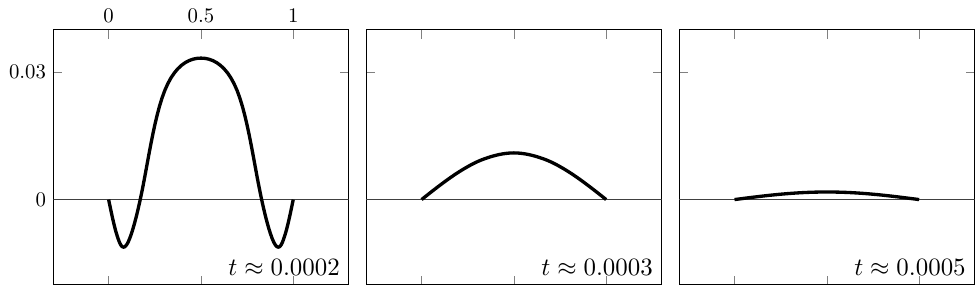}
        \caption{Example~\ref{ex:unconstrained_symmetric} with $\lambda = 900$ (vertically rescaled)}
        \label{fig:unconstrained_symmetric-2-rescaled}
    \end{figure}
\end{example}

The above examples suggest that the penalization parameter $\lambda$ strongly affects the possibility of migration, and more precisely it is hard to migrate if $\lambda$ is large. This would be compatible with the fact that the singular limit $\lambda=\infty$ formally corresponds to the curve shortening flow, for which no migration occurs.

\bibliography{ref_KM}

\end{document}